\documentclass[12pt]{article}

\usepackage{amsthm, amssymb, amstext}
\usepackage[fleqn]{amsmath}
\usepackage{amsthm}
\usepackage{todonotes}

\def\A{\mathcal{A}}

\def\F{\mathcal{F}}
\def\I{\mathcal{I}}

\pgfmathsetmacro\Rad{3.8}
\pgfmathsetmacro\rad{3}
\pgfmathsetmacro\wid{.3}
\pgfmathsetmacro\smal{\rad-\wid}
\pgfmathsetmacro\larg{\rad+\wid}
\def\aa{47}
\def\ab{73}
\def\ac{77}
\def\ad{133}
\def\ae{136}
\def\af{154}

\newtheorem{theorem}{Theorem}
\newtheorem{lemma}{Lemma}

\newtheorem{conjecture}{Conjecture}

\newtheorem{problem}{Problem}

\theoremstyle{definition}

\newtheorem{example}{Example}

\definecolor{gren}{RGB}{  0,140, 10}

\title{An Erd\H os-Ko-Rado Theorem for unions of length 2 paths}
\date{}
\author{
Carl Feghali  \thanks{Computer Science Institute of Charles University, Prague, Czech Republic, email: \texttt{feghali.carl@gmail.com} } \quad  
Glenn Hurlbert \thanks{Department of Mathematics and Applied Mathematics, Virginia Commonwealth University, USA, email: \texttt{ghurlbert@vcu.edu} }\\ 
Vikram Kamat \thanks{Department of Mathematics \& Statistics, Villanova University, Villanova, PA, USA, email: \texttt{vikram.kamat@villanova.edu}}}

\begin{document}   
\maketitle

\begin{abstract}
A family of sets is intersecting if any two sets in the family intersect. Given a graph~$G$ and an integer $r\geq 1$, let $\I^{(r)}(G)$ denote the family of independent sets of size $r$ of~$G$. For a vertex $v$ of $G$, the family of independent sets of size $r$ that contain~$v$ is called an $r$-star. Then $G$ is said to be $r$-EKR if no  intersecting subfamily of $ \I^{(r)}(G)$ is bigger than the largest $r$-star. Let $n$ be a positive integer, and let $G$ consist of the disjoint union of $n$ paths each of length 2. We prove that if $1 \leq r \leq n/2$, then $G$ is $r$-EKR. This affirms a longstanding conjecture of Holroyd and Talbot for this class of graphs and can be seen as an analogue of a well-known theorem on signed sets, proved using different methods, by Deza and Frankl and by Bollob\'as and Leader. 

Our main approach is a novel probabilistic extension of Katona's elegant cycle method, which might be of independent interest.   
\end{abstract}

\section{Introduction}

%Unless stated otherwise, we shall use small letters such as $x$ to denote non-negative integers or elements of a set,
%capital letters such as $X$ to denote sets, and calligraphic
%letters such as $\mathcal{F}$ to denote \emph{families}
%(sets whose members are sets themselves). %It is to be assumed that arbitrary sets and families are finite. 
%The set of positive integers is denoted by $\mathbb{N}$. 
The set $\{i \in \mathbb{N} \colon m \leq i \leq n\}$ is denoted by $[m,n]$, $[1,n]$ is abbreviated to $[n]$, and $[0]$ is taken to be the empty set $\emptyset$. For a set $X$, the \emph{power set of $X$} (that is, $\{A \colon A \subseteq X\}$) is denoted by $2^X$. 
The family of $r$-element subsets of $X$ is denoted by $\binom{X}{r}$. 
The family of $r$-element sets in a family $\mathcal{F}$ is denoted by $\mathcal{F}^{(r)}$. 
If $\mathcal{F} \subseteq 2^X$ and $x \in X$, then the family $\{A \in \mathcal{F} \colon x \in A\}$ is denoted by $\mathcal{F}(x)$ and called a \emph{star of $\mathcal{F}$} with \emph{centre} $x$. 
A family $\F$ is \emph{intersecting} if $F, F' \in \F$ implies $F \cap F' \not= \emptyset$. 

How large can a maximum-size intersecting family $\F \subseteq \binom{[n]}{r}$ be? If $2r > n$ then $|\F| = \binom{n}{r}$ is obvious, while if $2r \leq n$ the classical Erd\H{o}s--Ko--Rado (EKR) Theorem \cite{erdos} states that $\F$ can be no larger than a star. 

\bigskip

\noindent
\textbf{EKR Theorem} (Erd\H{o}s, Ko and Rado \cite{erdos}).
Let $n, r \geq 0$ be integers, $n \geq 2r$. Let $\F \subseteq \binom{[n]}{r}$ be intersecting. 
Then 
\begin{equation}\label{eq:ekr} |\F| \leq \binom{n - 1}{r - 1} = |\F(1)|.\end{equation}

\bigskip

When $n = 2r$, the proof of the EKR Theorem is easy: simply partition $\binom{[2r]}{r}$ into complementary pairs. Since $\F$ can contain at most one set from each pair, $|\F| \leq \frac{1}{2}\binom{2r}{r} = \binom{2r - 1}{r - 1}$. To deal with the case $n > 2r$ Erd\H{o}s, Ko and Rado \cite{erdos} introduced an important operation on families called \emph{shifting}.  

%If $\mathcal{F}(x) \neq \emptyset$, then we call $\mathcal{F}(x)$ a \emph{star of $\mathcal{F}$} with \emph{centre} $x$.

%One of the oldest and most fundamental results in extremal set theory is the Erd\H{os}-Ko-Rado (EKR) theorem \cite{erdos}, 
%which states that if $1 \leq r \leq n/2$ and $\A$ is an intersecting family of $r$-element subsets of $[n]$, then $|\A| \leq \binom{n - 1}{r - 1}$.  Notice that this bound can be achieved by taking the family of 
%$r$-element sets containing a particular element. Several proofs of the EKR Theorem have since been obtained (see \cite{D,FF2,HK,K,Kat}),  two of which are particularly remarkable: Katona's \cite{K}, which introduced the elegant cycle method, and Daykin's \cite{D}, using the fundamental Kruskal--Katona Theorem \cite{Ka,Kr}.

Let $\mathcal{I}_G$ denote the family of independent sets of $G$. %An independent set $J$ of $G$ is \emph{maximal} if $J \nsubseteq I$ for each independent set $I$ of $G$ such that $I \neq J$. The size of a smallest maximal independent set of $G$ is denoted by $\mu(G)$. 
The size of a maximum independent set of $G$ is denoted $\alpha(G)$. 
Holroyd and Talbot \cite{holroyd} introduced the problem of determining whether $\I^{(r)}_G$ has the star property  for a given graph $G$ and an integer $r \geq 1$. Following their terminology, graph $G$ is said to be $r$-EKR if no intersecting family of $\I^{(r)}_G$ is bigger than the largest star of $\I^{(r)}_G$.

Although not phrased in the language of graphs, one of the earliest results in the area, proved using different methods by Deza and Frankl \cite{deza} and by Bollob\'as and Leader \cite{leader}, was to show that if $G$ is the vertex-disjoint union of $n$ complete graphs each of size $k \geq 2$, then $G$ is $r$-EKR ($1 \leq r \leq n$). This result was extended in various ways \cite{signed1, signed2, signed3, signed4, signed5}. One such extension that is directly relevant to us is given by Hilton and Spencer \cite{spencer3, spencer2}, showing that if $G$ is the vertex-disjoint union of powers of cycles or of a power of a path and powers of cycles, then $G$ is $r$-EKR ($1 \leq r \leq \alpha(G)$), provided some condition on the clique number is satisfied (see \cite{borgfeghali} for short proofs with somewhat weaker bounds). The problem, however, of obtaining an EKR result for vertex-disjoint unions of (powers of) paths remained elusive. In this note, we make the first step towards this problem in the following theorem. 
 
\begin{theorem}\label{thm:main}
Let $2r \leq n$, and let $G$ be the vertex-disjoint union of $n$ paths each of length 2. Then $G$ is $r$-EKR. 
\end{theorem}

We remark that Theorem \ref{thm:main} verifies a conjecture of Holroyd and Talbot \cite{holroyd2} for vertex-disjoint unions of length 2 paths. It is also the first case where the graph is not transitive, cf.\ \cite{wang}.  

Our most important message, however, is the technique used to establish Theorem \ref{thm:main}. Namely, we shall use the \textit{cycle method}, a technique first introduced by Katona \cite{katonaekr} in his beautiful proof of the EKR Theorem; however, some difficulties which are not present in \cite{katonaekr} must be dealt with. Roughly speaking, we combine the shifting technique with a weighted (or probabilistic) version of the cycle method by considering pairs of intervals  around the circle, instead of intervals, in order to account for the different types of vertices in the graph.

\section{Preliminaries}

%The proof consists of a probabilistic argument akin to Katona's proof of the EKR theorem.  Let us first introduce some notation and preliminary lemmas. 

Throughout the rest of the paper, let $n$ be a positive integer, and let $G$ be the vertex-disjoint union of $n$ paths $P_1, \dots, P_n$ each of length 2. For $i \in [n]$, the vertex set and edge set of $P_i$ are, respectively, $\{x_i, y_i, z_i\}$ and $\{x_iy_i, y_iz_i\}$. Define $X = \{x_i: 1 \leq i \leq n\}$, 
$Y = \{y_i: 1 \leq i \leq n\}$, 
$Z = \{z_i: 1 \leq i \leq n\}$, and 
$L = X \cup Z$.

To prove the theorem, we will use the powerful shifting technique, first introduced by Erd\H{o}s, Ko and Rado \cite{erdos}. A family $\F \subseteq \I^{(r)}_G$ is said to be \emph{shifted} if $F \in \F$ and  $y_i \in F \cap Y$ implies $(F \setminus \{y_i\}) \cup \{x_i\} \in \F$. 

Let $\phi_i: V(G) \rightarrow V(G)$ be the function given by $$
\phi_i(y_i) = x_i \mbox{ for $i \in [n]$,} \quad  \phi_i(v) = v \mbox{ otherwise.} 
$$
For $A \subseteq V(G)$ and $i\in [n]$, let 
$$\phi_i(A)=\{\phi_i(a): a\in A\},$$
and note that if $A$ is independent then so is $\phi_i(A)$. 

For a family $\F \subseteq \I^{(r)}_G$, let $\Phi_i: \I^{(r)}_G \rightarrow \I^{(r)}_G$ be given by
$$\Phi_i(\F) = \{\phi_i(A): A \in \F \} \cup \{A: A, \phi_i(A) \in \F\},$$
and let $\Phi(\F) = \Phi_1(\F) \circ \dots \circ \Phi_n(\F)$. Informally speaking, for each $A \in \F$ that contains vertex $y_i$, $\Phi_i(\F)$ replaces $A$ by the independent set $A^{'} = (A \setminus \{y_i\}) \cup \{x_i\}$ provided $A^{'}$ is not already in $\F$, and $\Phi(F)$ performs this operation for each $i\in [n]$. It is easy to see that $\Phi(\F)$ is shifted.  

\begin{lemma}\label{lem:shift}
Let $\F \subseteq \I^{(r)}_G$ be an intersecting family. Then $|\Phi(\F)| = |\F|$ and $A \cap B \cap L \not=\emptyset$ for all $A, B \in \Phi(\F)$. 
\end{lemma}

\begin{proof}
The proof is completely standard, but we include the details for completeness. We first demonstrate that $\Phi_i(\F)$ is intersecting for $i \in [n]$. 

Let $A, B \in \Phi_i(\F)$. 
If $A, B \in \F$, then $A \cap B \not=\emptyset$ since $\F$ is intersecting.   
If $A, B \in \Phi_i(\F) - \F$, then by definition $x_i \in A \cap B$. 
So we can assume that $A \in \Phi_i(\F) \cap \F$ and $B \in \Phi_i(\F) - \F$. 
Then $B = (C - \{y_i\}) \cup \{x_i\}$ for some $C \in \F$ and either $y_i \not\in A$ or 
$D = (A - \{y_i\}) \cup \{x_i\}$ for some $D \in \F$. If $y_i \not\in A$, 
then $A \cap B = A \cap ((C - \{y_i\}) \cup \{x_i\}) \supseteq A \cap C \not=\emptyset$ 
since $A, C \in \F$. Finally, if $D \in \F$, then $\emptyset \not=C \cap D= ((B - \{x_i\}) \cup \{y_i\})  \cap ((A - \{y_i\}) \cup \{x_i\})  = A \cap B$. This shows that $\Phi_i(\F)$ is intersecting for $i \in [n]$. 

By definition, $|\Phi(\F)| = |\F|$. Since each $\Phi_i(\F)$ is intersecting, $\Phi(\F)$ is intersecting.  We are left to show that $A \cap B \cap L \not=\emptyset$ for all $A, B \in \Phi(\F)$. Suppose that $A \cap B = \{y_{j_1}, \dots, y_{j_t}\} \subseteq G \setminus L$ for some $1 \leq t \leq n$. By definition, $A' = (A \setminus \{y_{j_1}, \dots, y_{j_t}\}) \cup \{x_{j_1}, \dots, x_{j_t}\}$ is a member of $\Phi(\F)$. But then $A' \cap B = \emptyset$, contradicting that $\Phi(\F)$ is intersecting. 
\end{proof}

Another tool in our proof of the theorem is the cycle method of Katona. Before we can concisely use the method in the next section, here we only make a few definitions and prove a lemma. 
We shall find it convenient to represent $L$ by $[2n]$, by representing the vertex $x_i$ by the integer $i$ and $z_i$ by $n + i$ for $i \in [n]$.
We use the permutation notation $\sigma=\langle a_1,...,a_{2n}\rangle$ to mean that the elements of $L$ are listed in the order $a_1, ...,a_{2n}$ around the circle. 
In this case we write $\sigma(i)=a_i$.
For $j \geq 1$, $M \subseteq L$, and a permutation $\sigma$ of $L$, let 
\begin{equation}
\label{e:sigmaMj}
    _{\sigma}M^j = \{ \sigma(c + j): c \in M\}.
\end{equation}
\begin{example}
Let $L = \{1, 2, 3, 4, 5, 6\}$, $\sigma = \langle 3, 5, 6, 1, 2, 4\rangle$ and $M = \{1, 4, 5\}$. Then one has, for example,
${_{\sigma}M^1} = \{\sigma(1 + 1), \sigma(4 + 1), \sigma(5 + 1)\} = \{5, 2, 4\}$ and 
${_{\sigma}M^3} = \{\sigma(1 + 3), \sigma(4 + 3), \sigma(5 + 3)\} =  \{1, 3, 5\}$. 
\end{example}

Set $_{\sigma}M^0 = {_{\sigma}M}$ and ${_{\langle 1, \dots, 2n\rangle}M} = M$.  Following Bollob\'as and Leader \cite{leader}, call a permutation $\sigma$ of $L$ \emph{good} if $x_i$ and $z_i$ are diametrically opposite on the circle.    Equivalently, $\sigma$ is good if  any $n$ elements $a_1, \dots, a_n$ in $L$ appearing consecutively in $\sigma$ do not contain both $x_i$ and $z_i$ for each $i \in [n]$. 
(For instance, the permutation $\langle 1, \dots, 2n\rangle$ is a good permutation of $L$.)
%For $M \subseteq L$, $\sigma \in \mathcal{L}$ and $j \in [2n]$, define
%$$
%\Gamma({_{\sigma}M^j}) = \Big|\{\pi \in \mathcal{L} \mid {_{\pi}M^j}  = {_{\sigma}M^j} \}\Big|. 
%$$

%\begin{lemma}\label{lem:gamma}
%For any $M \subseteq L$, $\sigma, \pi \in \mathcal{L}$ and $j, k \in [2n]$, we have \begin{equation}\label{eq:same}
%\Gamma({_{\sigma}M^j}) = \Gamma({_{\pi}M^k})
%\end{equation}
%\end{lemma}

%\begin{proof}
%Let $m = |M|$ and let $s$ be the number of pairs of siblings in $M$. Using that $\sigma$ is good, it is a small exercise to show that $\Gamma({_{\sigma}M^j}) = (m - s)!(n - m + s)!$. The lemma immediately follows. 
%\end{proof}

We now prove a lemma analogous to Katona's lemma in his proof of the EKR Theorem. 
  
For integers $s \geq 0$ and $t, u \geq 1$, define the pairwise disjoint intervals 
\begin{gather*}
S = [1, \dots, s], \\
T_1 = [s + 1, \dots, s + t], \\
U_1 = [s + 1 + n, \dots, s + u + n], \\
T_2 =  [s+u+n+1, \dots, s+u+n+t],\\
U_2 = [s + t + 1, \dots, s + t + u].
\end{gather*}
%\gh{I think that the definition of $T_2$ should be $$, since $T_2$ follows $U_1$ on the circle.}\\

%\begin{gather*}
%T_1 = [a_1, \dots, a_t], \\
%U_1 = [a_{1+ n}, \dots, a_{u + n}], \\
%T_2 =  [a_{t + 1 + n}, \dots, a_{2t + n}],\\
%U_2 = [a_{t + 1}, \dots, a_{t + u}].
%\end{gather*}

Let  $C_i(t, u) = T_i \cup U_i $ for $i \in \{1, 2\}$. 
For a good permutation $\sigma$ of $L$, recall the notation of equation (\ref{e:sigmaMj}) and define 
$$_{\sigma}\mathcal{C}(t, u) = \{_{\sigma}C_1^j(t, u): 1 \leq j \leq 2n\} \cup \{_{\sigma}C_2^j(t, u): 1 \leq j \leq 2n\}.$$ 

%In words, $M^j$ is obtained from $M$ by $j$ clockwise rotations along $\sigma$.
% and write $C_i^j(t,u)$ and $\mathcal{C}(t,u)$ instead.

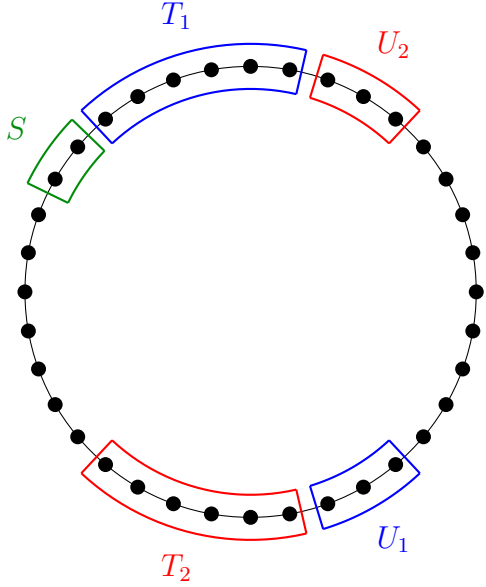
\begin{figure}
\centering
\begin{tikzpicture}
    %\tikzstyle{every node}=[draw,circle,fill=black,minimum size=1pt,inner sep=2pt]
    \foreach \x in {0,...,35} 
        \draw ({\rad*cos(10*\x)},{\rad*sin(10*\x)}) node[circle,fill=black,inner sep=2pt] {};
    \draw circle (\rad cm);

\draw [red,thick,domain=\aa:\ab] plot ({\larg*cos(\x)}, {\larg*sin(\x)});
\draw [red,thick,domain=\aa:\ab] plot ({\smal*cos(\x)}, {\smal*sin(\x)});
\draw [red,thick,domain=\smal:\larg] plot ({\x*cos(\aa)}, {\x*sin(\aa)});
\draw [red,thick,domain=\smal:\larg] plot ({\x*cos(\ab)}, {\x*sin(\ab)});
\draw [red] ({\Rad*cos((\aa+\ab)/2)},{\Rad*sin((\aa+\ab)/2)}) node[label=center:$U_2$] {};

\draw [blue,thick,domain=\ac:\ad] plot ({\larg*cos(\x)}, {\larg*sin(\x)});
\draw [blue,thick,domain=\ac:\ad] plot ({\smal*cos(\x)}, {\smal*sin(\x)});
\draw [blue,thick,domain=\smal:\larg] plot ({\x*cos(\ac)}, {\x*sin(\ac)});
\draw [blue,thick,domain=\smal:\larg] plot ({\x*cos(\ad)}, {\x*sin(\ad)});
\draw [blue] ({\Rad*cos((\ac+\ad)/2)},{\Rad*sin((\ac+\ad)/2)}) node[label=center:$T_1$] {};

\draw [gren,thick,domain=\ae:\af] plot ({\larg*cos(\x)}, {\larg*sin(\x)});
\draw [gren,thick,domain=\ae:\af] plot ({\smal*cos(\x)}, {\smal*sin(\x)});
\draw [gren,thick,domain=\smal:\larg] plot ({\x*cos(\ae)}, {\x*sin(\ae)});
\draw [gren,thick,domain=\smal:\larg] plot ({\x*cos(\af)}, {\x*sin(\af)});
\draw [gren] ({\Rad*cos((\ae+\af)/2)},{\Rad*sin((\ae+\af)/2)}) node[label=center:$S$] {};

\draw [blue,thick,domain=360-\ab:360-\aa] plot ({\larg*cos(\x)}, {\larg*sin(\x)});
\draw [blue,thick,domain=360-\ab:360-\aa] plot ({\smal*cos(\x)}, {\smal*sin(\x)});
\draw [blue,thick,domain=\smal:\larg] plot ({\x*cos(360-\aa)}, {\x*sin(360-\aa)});
\draw [blue,thick,domain=\smal:\larg] plot ({\x*cos(360-\ab)}, {\x*sin(360-\ab)});
\draw [blue] ({\Rad*cos(360-((\aa+\ab)/2))},{\Rad*sin(360-((\aa+\ab)/2))}) node[label=center:$U_1$] {};

\draw [red,thick,domain=360-\ad:360-\ac] plot ({\larg*cos(\x)}, {\larg*sin(\x)});
\draw [red,thick,domain=360-\ad:360-\ac] plot ({\smal*cos(\x)}, {\smal*sin(\x)});
\draw [red,thick,domain=\smal:\larg] plot ({\x*cos(360-\ac)}, {\x*sin(360-\ac)});
\draw [red,thick,domain=\smal:\larg] plot ({\x*cos(360-\ad)}, {\x*sin(360-\ad)});
\draw [red] ({\Rad*cos(360-((\ac+\ad)/2))},{\Rad*sin(360-((\ac+\ad)/2))}) node[label=center:$T_2$] {};

\end{tikzpicture}
\caption{The sets $S$, $T_1$, $U_1$, $T_2$, and $U_2$, with $s=2$, $t=6$, and $u=3$, where $n=18$.
$C_1(6,3)$ is in blue and $C_2(6,3)$ is in red.}
\label{circle_pic}
\end{figure}

\begin{example}
Consider the parameter values in Figure \ref{circle_pic}, and let
\begin{eqnarray*}
\sigma=
&\langle 5, 8,21, 6,20, 1,11,14,36,22,10,34,30,27, 7,15,31,17,\\
& 23,26, 3,24, 2,19,29,32,18, 4,28,16,12, 9,25,33,13,35\rangle.
\end{eqnarray*}
Then ${_\sigma}\mathcal{C}(6,3)$ consists of the following sets.
$$\begin{array}{rclrcl}
{_\sigma}C_1^1&=&\{ 6,20, 1,11,14,36,  24, 2,19\}&
{_\sigma}C_2^1&=&\{29,32,18, 4,28,16,  22,10,34\}\\
{_\sigma}C_1^2&=&\{20, 1,11,14,36,22,   2,19,29\}&
{_\sigma}C_2^2&=&\{32,18, 4,28,16,12,  10,34,30\}\\
{_\sigma}C_1^3&=&\{ 1,11,14,36,22,10,  19,29,32\}&
{_\sigma}C_2^3&=&\{18, 4,28,16,12, 9,  34,30,27\}\\
&\cdot&&&\cdot&\\
&\cdot&&&\cdot&\\
&\cdot&&&\cdot&\\
{_\sigma}C_1^{36}&=&\{21, 6,20, 1,11,14,   3,24, 2\}&
{_\sigma}C_2^{36}&=&\{19,29,32,18, 4,28,  36,22,10\}\\
\end{array}$$
\end{example}

\begin{lemma}\label{lem:katona}
Let $n, t, u \geq 0$ be integers such that  $t \geq u $. Let $\sigma$ be a good permutation of $L$.  For any intersecting family $\mathcal{B} \subseteq {\mathcal{C}(t, u)} := {_\sigma {\mathcal{C}(t, u)}}$, 

\begin{itemize}
\item[(i)]
$|\mathcal{B}| \leq t$ and  $|{_\sigma \mathcal{C}(t, u)}| = 2n$ if $u = 0$ and and $n \geq t$;
\label{lem:kat1}

\item[(ii)]
$|\mathcal{B}| \leq t$ and  $|{_\sigma \mathcal{C}(t, u)}| = n$ if $t = u$ and $n \geq 2t$;
\label{lem:kat2}

\item[(iii)]
$|\mathcal{B}| \leq 2(t + u)$ and  $|{_\sigma \mathcal{C}(t, u)}| = 4n$ if $t > u \geq 1$ and $n \geq 2(t + u)$. 
\label{lem:kat3}

\end{itemize}

\end{lemma}

\begin{proof}
Since $\sigma$ is fixed we will write $C_i^j(t,u)$ in place of ${_\sigma}C_i^j(t,u)$ for visual simplicity.

(i) This case is Katona's lemma. Clearly, ${C_1^j(t, 0)} = {C_2^{j + n}(t, 0)}$ and hence $\mathcal{C}(t, 0) = \{C_1^j(t, 0): 1 \leq j \leq 2n\}$, which implies $|{\mathcal{C}(t, u)}| = 2n$. Now assume without loss of generality that $C_1^1(t, 0) \in \mathcal{B}$. All the other sets $C_1^i(t, 0)$ that intersect $C_1^1(t, 0)$ can be partitioned into disjoint pairs $({C_1^i(t, 0)}, {C_1^{i - t}(t, 0))}$ for $i \in [2, t]$ since $n \geq t$. Since $\mathcal{B}$ is intersecting, it can contain at most one set from each pair. Hence $|\mathcal{B}| \leq t$. 

(ii) Clearly, $C_1^j(t, t) = {C_2^{j - t}(t, t)} = {C_1^{j + n}(t, t)}$ and hence $\mathcal{C}(t, t) = \{C_1^j(t, t): 1 \leq j \leq n\}$, which implies $|{\mathcal{C}(t, u)}| = n$. Since $n \geq 2t$, it follows as in (i) that $|\mathcal{B}| \leq t$. 

(iii) Clearly,  ${C_1^i(t, u)} \not= {C_2^j(t, u)}$ for any $1 \leq i, j \leq 2n$, implying  $|{\mathcal{C}(t, u)}| = 4n$. Now let $D^i = {C_1^i(t, u)} \cup {C_2^i(t, u)}$. Then \begin{equation}\label{eq:d}
D^i = D^j \mbox{ if and only if } j = n + i. 
\end{equation}

Define the family 
$$\mathcal{D}' = \{D^i: {C_j^i(t, u)} \in \mathcal{B} \mbox{ for some } j \in \{1, 2\}\}.$$
Then $\mathcal{D}'$ is intersecting since $\mathcal{B}$ is intersecting. Since $n \geq 2(t + u)$, it follows by (ii) that $|\mathcal{D}'| \leq t + u$. Now  $\mathcal{B}$ can contain at most two of the sets in $\{{C_1^i(t, u)}, {C_2^i(t, u)}, {C_1^{n + i}(t, u)}, {C_2^{n + i}(t, u)} \}$ since it is intersecting. By (\ref{eq:d}), $|\mathcal{B}| \leq 2|\mathcal{D}'| \leq 2(t + u)$ as required. 
\end{proof}

%We end this section with the following lemma, whose simple proof is left to the reader. 

%\begin{lemma}\label{lem:equal}
%Let $A \in \mathcal{I}^{(r)}_G$, let $\ell = |A \cap L|$ and let $s$ be the number of pairs of siblings of $A$. For each $i \in \{1, 2\}$ and $j \in [2n]$ there exists $\sigma \in \mathcal{L}$ such that 
%$
%{_{\sigma}C_i^j(t, u)} = A \cap L 
%$
% if and only if $t = \ell - s$ and $u = s$. \end{lemma}

%In human language, a set $T \in {M_{\pi}(i) \choose t}$ is antisymmetric if $T$ cannot be obtained by $j$ clockwise rotations of $M_{\pi}(i) \setminus T$ for any $1 \leq j \leq 2n$. 

\section{The proof of Theorem \ref{thm:main}}
Before we begin the proof of the theorem, we introduce some additional notation. For a vertex $\ell \in L$, we let $\zeta(\ell) \in Y$ denote the unique neighbour of $\ell$ in $G$. We say that $x_i$ and $z_i$ are \emph{siblings}. Given a good permutation $\sigma = \langle a_1, \dots, a_{2n}\rangle$ of $L$, we let $\sigma' = \langle\zeta(a_1), \dots,\zeta(a_{2n})\rangle$. We say that an independent set $A$ of $G$ is \begin{itemize}
\item of \emph{type I} if whenever $A \cap L$ contains an element of $X \cup Z$, then it does  not contain its sibling,
\item of \emph{type II} if whenever $A \cap L$ contains an element of $X \cup Z$, then it contains its sibling, and
\item of \emph{type III} in all other cases.   
\end{itemize}
Let $\F \subseteq \I^{(r)}_G$ be an intersecting family.  By Lemma \ref{lem:shift}, to prove the theorem we can assume that $\F$ is shifted.  

For $0 \leq s \leq r$,  define $\I^{(r)}_G(s) = \{A \in \I^{(r)}_G: |A \cap Y| = s \}$ and
$\F_s = \{F \in \F: |F \cap Y| = s\}$.
Since $\F$ is shifted,  $\F_r = \emptyset$ and so
$|\F| = \sum_{s = 0}^{r - 1} |\F_s|.$ 
Let us now try to  bound the size of each $\F_s$. 

Define the families
\begin{gather*}
\mathcal{A}_s = {\bigcup_{t \in [ \lfloor \frac{r - s}{2} \rfloor + 1, \dots, r - s - 1]} \Big\{ C_1(t, r - s - t)\}, C_2(t, r- s- t)\Big\}}, \\
\mathcal{B}_s =\bigg\{\begin{array}{ll} \{C_1(r - s, 0)\} \cup \{ C_1(\frac{r - s}{2}, \frac{r - s}{2})\} & \mbox{ if $r - s$  is even, }\\ \{C_1(r - s, 0)\},& \mbox{otherwise,}\end{array}
\end{gather*}

Now, choose an index $i \in [2n]$ and a good permutation $\sigma$ of $L$ uniformly and independently at random and  a member $C \in \mathcal{A}_s \cup \mathcal{B}_s$ with probability \begin{equation*}\label{eq:prob}
h(C) := f(C) \Big/ \sum_{D \in \mathcal{A}_s \cup \mathcal{B}_s} f(D),
\end{equation*} where $f$ will be determined later on. 
Recalling again the notation in definition (\ref{e:sigmaMj}), we set 
$$
I ={_{\sigma'}S^i} \cup  {_{\sigma}C^{i}}.
$$

Clearly $I \in \mathcal{I}^{(r)}_G(s)$. To ensure that $I$ is uniformly chosen from $\mathcal{I}^{(r)}_G(s)$, consider arbitrary $K \in \mathcal{I}^{(r)}_G(s)$. What is the probability that $I = K$? It is the probability that $ {_{\sigma}C^{i}} = K \cap L$ and ${_{\sigma'}S^i} = K \cap Y$. In particular, let $k_1 = |K \cap L| = r - s$ and let $k_2$ be the number of pairs of siblings in $K \cap L$. Then $ {_{\sigma}C^{i}} = K \cap L$ only when $k_2$ is the number of pairs of siblings in $C$: there is precisely one such $C$ in the situation when $k_1 = 2k_2$ or $k_2 = 0$ and one such $C$ together with its `complement' $C^\star$ in all other cases, where $C^*=C_{3 - i}(t, u)$ when $C=C_{i}(t, u)$. 
Hence, with $\mathcal{L}$ equal to the set of all good permutations of $L$, conditioning on $i$ the former probability is equal to
$$\frac{(k_1 - k_2)!(n - k_1 + k_2)!}{|\mathcal{L}|}\cdot h(C) \mbox{ if $K$ is of type I or II, and}$$ 
$$\frac{(k_1 - k_2)!(n - k_1 + k_2)! }{|\mathcal{L}|} \cdot \bigg(h(C)+ h(C^\star)\bigg) \mbox{ if $K$ is of type III}$$  while the latter probability is invariably equal to $$\frac{2^s(n - s)!s!}{|\mathcal{L}|},$$ 
Therefore, by taking $f(C) = \frac{1}{(k_1 - k_2)!(n - k_1 + k_2)!}$ if $I$ is of type I or II and $f(C) = f(C^{\star}) = \frac{1}{2(k_1 - k_2)!(n - k_1 + k_2)!}$ if $I$ is of type III, we have that $I$ is uniformly chosen from $\mathcal{I}^{(r)}_G(s)$. This means that
%Applying Lemma \ref{lem:gamma}, $I$ is uniformly chosen from $\mathcal{I}^{(r)}_G(s)$ and hence
\begin{equation}\label{eq:f}
\Pr[I \in \F_s] = \frac{|\F_s|}{|\mathcal{I}^{(r)}_G(s)|} = \frac{|\F_s|}{\binom{n}{s}\binom{2n - 2s}{r - s}}.
\end{equation}

We next turn to estimating $\Pr[I \in \F_s]$ in another way.  For a good permutation $\sigma$ of $L$ and special sets $C_1, \dots, C_t$, write
$$\delta(I, \sigma, C_1, \dots, C_t): = \Pr[I \in \F_s \mid \sigma \land (C_1 \lor \dots \lor C_t)] \cdot \Pr[\sigma \land (C_1 \lor \dots \lor C_t)].$$

Using Lemma \ref{lem:shift} and either Lemma \ref{lem:katona}(i) or Lemma \ref{lem:katona}(ii), for each member $B$ of $\mathcal{B}_s$ and each good permutation $\sigma$ of $L$,
\begin{equation}\label{eq:fff1}
\Pr[I \in \F_s \mid \sigma \land B] \leq \frac{r - s}{2n}.\end{equation}

Similarly, using Lemma \ref{lem:shift} in conjunction with Lemma \ref{lem:katona}(iii), for each member $A$ of $\mathcal{A}_s$ and good permutation $\sigma$ of $L$,
\begin{equation}\label{eq:fff2}
\Pr[I \in \F_s \mid \sigma \land (A \lor A^\star)] \leq \frac{2(r - s)}{4n} = \frac{r - s}{2n}.\end{equation} 

Let $\mathcal{A}'_s = \{(A, A^{\star}): A, A^{\star} \in \A_s\}$. By (\ref{eq:fff1}), (\ref{eq:fff2}) and the law of total probability, we have that
\begin{eqnarray}\label{eq:delta}
\Pr[I \in \F_s] 
&=&\sum_{\sigma \in \mathcal{L}, B \in \mathcal{B}_s}  \delta(I, \sigma, B) +  \sum_{\sigma \in \mathcal{L}, (A, A^{\star}) \in \A'_s} \delta(I, \sigma, A, A^\star) \nonumber \\
&\leq&\frac{r - s}{2n}\Bigg( \sum_{\sigma \in \mathcal{L}, B \in \mathcal{B}_s} \Pr[\sigma \land B] +  \sum_{\sigma \in \mathcal{L}, (A, A^{\star}) \in \A'_s} \Pr[\sigma \land (A\lor A^\star)]\Bigg) \nonumber \\
&=&\frac{r - s}{2n}\Bigg( \sum_{\sigma \in \mathcal{L}, B \in \mathcal{B}_s} \Pr[\sigma \land B] +  \sum_{\sigma \in \mathcal{L}, A \in \mathcal{A}_s} \Pr[\sigma \land A]\Bigg) \nonumber \\
&=& \frac{r - s}{2n}\sum_{\sigma \in \mathcal{L}, C \in \mathcal{A}_s \cup \mathcal{B}_s} \Pr[\sigma \land C]  \nonumber \\ &=& \frac{r - s}{2n}. 
\end{eqnarray}

Combining (\ref{eq:f}) and (\ref{eq:delta}) immediately yields the result, as follows. 

\begin{align*}
    |\F|
    &= \sum_{s=0}^{r-1} |\F_s|\\
    &\le \sum_{s=0}^{r-1}\frac{r-s}{2n}\binom{n}{s}\binom{2n-2s}{r-s}\\
    &= \sum_{s=0}^{r-1}\binom{n-1}{s}\binom{2n-2s-1}{r-s-1}\\
    &=|\F(x_1)|. \qed
\end{align*}

\section{Concluding Remarks}

Our bound on $r$ in Theorem \ref{thm:main} is probably not optimal. In view of the EKR theorem, we make the following conjecture. 

\begin{conjecture}
Let $r \leq n$, and let $G$ be the vertex-disjoint union of $n$ paths each of length $2$. Then $G$ is $r$-EKR. 
\end{conjecture}

Perhaps another way to generalise Theorem \ref{thm:main} is to consider a larger class of related graphs. To be more precise, for a positive integer $k$, we define a \emph{$k$-claw} to be the tree in which the root has degree $k$ and every other vertex degree one. For example, a $2$-claw is simply a path of length 2. 

\begin{problem}\label{problem1}
For any integers $k, n, r$, such that $k \geq 3$ and $2r \leq n$, show that if $G$ is the vertex-disjoint union of $n$ $k$-claws, then $G$ is $r$-EKR.  
\end{problem}

Although we suspect that the approach in this paper can be adapted to address Problem \ref{problem1}, there are a few complications that arise which do not allow for straightforward alterations. 

One of the interesting consequences of studying the Erd\H{o}s--Ko--Rado question for families of independent sets of graphs is the connection to a famous, longstanding conjecture of Chv\'atal \cite{chvatalc} concerned with \textit{hereditary} families of subsets of a finite set. A family $\mathcal{H}\subseteq 2^{[n]}$ is hereditary if $A\in \mathcal{H}$ and $B\subseteq A$ implies that $B\in \mathcal{H}$. The conjecture then states that no intersecting subfamily of $\mathcal{H}$ is larger than the largest star of $\mathcal{H}$.

It is easy to see that for graph $G$ and $r\geq 1$, the family of all independent sets of size at most $r$, denoted by $\mathcal{I}^{(\leq r)}_G$, is a hereditary family. A simple corollary of  \ref{thm:main} is that Chv\'atal's conjecture is true for $\mathcal{I}^{(\leq r)}_G$, where $G$ is the vertex-disjoint union of $n$ paths, each of length $2$, and $r\leq n/2$. Conversely, a corollary of Chv\'atal's  conjecture is that for any graph $G$, no intersecting family of independent sets of $G$ is larger than the largest intersecting family of independent sets of $G$ each containing some particular vertex.

\section*{Acknowledgements}

The authors are indebted to both referees for several helpful suggestions which greatly improved the presentation. 
Carl Feghali was supported by grant 19-21082S of the Czech Science Foundation.

\bibliography{bibliography}{}
\bibliographystyle{abbrv}

\end{document}